\documentclass[11pt]{article}

\usepackage{subfigure}
\usepackage{color}
\usepackage[cmex10]{amsmath}
\usepackage{mathrsfs,amssymb,amsmath}
\usepackage{cases}
\usepackage{graphicx}
\usepackage{caption}

\usepackage{epstopdf}
\usepackage{hyperref}

\usepackage{float}
\usepackage{placeins}

\newcommand{\R}{{\mathbb R}}

\def\dref#1{(\ref{#1})}

\usepackage{algorithm} 
\usepackage{algorithmic} 

\newtheorem{corollary}{Corollary}[section]
\newtheorem{remark}{Remark}[section]
\newtheorem{lemma}{Lemma}[section]

\newtheorem{theorem}{Theorem}[section]
\newtheorem{definition}{Definition}[section]

\newenvironment{proof}{{\bf {Proof.}}}{\hfill $\square$}
\allowdisplaybreaks[4]

\numberwithin{equation}{section}

\def\dref#1{(\ref{#1})}

\def\pt{\partial}
\def\mc{\mathcal}
\def\ra{\rightarrow}

\def\s{\subseteq}

\def\e{\varepsilon}

\def\ol{\overline}

\def\vp{\varphi}

\def\bf{\textbf}
\def\pt{\partial}

\def\om{\omega}
\def\Om{\Omega}

\def\al{\alpha}
\def\be{\beta}
\def\de{\delta}
\def\ga{\gamma}

\def\Ga{\Gamma}
\def\La{\Lambda}

\def\ts{\times}

\def\iy{\infty}

\def\f{\frac}

\def\se{\setminus}

\def\ura{\rightharpoonup}

\def\df{\mathrm d}

\def\wt{\widetilde}
\def\wh{\widehat}

\def\essinf{\operatorname*{ess\ \! inf}}

\def\hra{\hookrightarrow}

\def\mcO{\mathcal{O}}

\def\mcD{\mathcal{D}}

	\DeclareMathOperator{\Div}{div}

	\DeclareMathOperator{\dist}{dist}

	\DeclareMathOperator{\supp}{{supp}}

	\newcommand{\N}{\mathbb N}

\begin{document}

\title{{\bf   {\bf  A Shape Design Approximation for Degenerate Partial Differential Equations and Its Application}}\footnote{\small This work was carried out with the support of the
National Natural Science Foundation of China under grant  nos. 12131008 and U23B2033, and
National Key R\&D Program of China under grant no. 2024YFA1013101.}}

\author{ Bao-Zhu Guo$^{a}$ and Dong-Hui Yang$^{b}$\footnote{\small
The corresponding author. Email: donghyang@139.com}   and   Jie Zhong$^{d}$
\\
$^a${\it Academy of Mathematics and Systems Science, Academia Sinica, Beijing 100190, China}\\
$^b${\it School of Mathematics and Statistics, Central South University}\\
			{\it Changsha 410075, P.R.China}\\
$^d${\it   Department of Mathematics,  California State University Los Angeles, Los Angeles, 90032, USA}}

\date{}

	\maketitle{}
	
	
	\begin{abstract}

  In this paper, we focus on two types of degenerate partial differential equations: a degenerate elliptic equation and a degenerate parabolic equation. Significantly, both categories are characterized by the same principal operator. To obtain solutions for these equations, we introduce a novel approximation approach, termed the shape design approximation. As a practical application of this method, we derive a Carleman estimate for the backward degenerate parabolic equation. This estimate plays a pivotal role in establishing the null controllability of the degenerate parabolic equation. A notable advantage of employing the shape design approximation in deriving the Carleman estimate is that it enables us to bypass the requirement for second  order derivatives in the degenerate equation. Usually, this has been a significant obstacle in the derivation of Carleman estimates for degenerate parabolic equations.

\vspace{0.3cm}

	
\vspace{0.3cm}
		

	\end{abstract}

	\maketitle{}
	\thispagestyle{empty}
	
	
	\section{Introduction}
	
	 Degenerate partial differential equations have been extensively studied in numerous works, including \cite{CS2,CS3,Fabes,GC,Heinonen,Mamedov,Pingen,Stuart,Trudinger}. The solution spaces for these equations are typically formulated within the framework of weighted Sobolev spaces (as detailed in \cite{GC,Heinonen}), which are contingent upon the so-called $A_p$ weight function.
To elaborate, for a given $p\in (1,\infty)$, a locally integrable, non-negative function $w$ is said to be an $A_p$-weight if there exists a constant $c>0$ such that, for all cubes $K\subset\mathbb{R}^N$, the following inequality holds:
\begin{equation}\label{01.01.1}
\left(\frac{1}{|K|}\int_Kw(x)\,dx\right)\left(\frac{1}{|K|}\int_Kw(x)^{-\frac{1}{p-1}}\,dx\right)^{p-1}\leq c.
\end{equation}
The infimum of the set of constants $c=c(w,p,N)>0$ satisfying \eqref{01.01.1} is referred to as the $A_p$-constant of $w$. For $p=1$, a locally integrable, non-negative function $w$ is designated an $A_1$-weight if there exists a constant $c>0$ such that, for all cubes $K\subset\mathbb{R}^N$, it has
\begin{equation*}
\frac{1}{|K|}\int_K w(x)\,dx\leq c\essinf_{x\in K}w(x).
\end{equation*}
It is well established that $w=|x|^\alpha\ (x\in\mathbb{R}^N)$ constitutes an $A_p$-weight if and only if $-N<\alpha<N(p-1)$. In particular, $x^\alpha\ (0<x<1)$ is an $A_{1+\frac{2}{N}}$-weight for $\alpha\in (0,2)$. In this study, we assume that $\Omega\subset\mathbb{R}^N\ (N\geq 1)$ is a bounded domain with a $C^2$-boundary, and $w\in C^1(\Omega)\cap C(\overline{\Omega})$ is a weight function belonging to the Muckenhoupt class $A_{1+\frac{2}{N}}$ (i.e., $w$ is an $A_{1+\frac{2}{N}}$-weight), satisfying
\begin{equation}\label{01.01.7}
w(x)>0 \text{ in } \Omega \hbox{ and } w(x)=0 \text{ on } \Gamma,
\end{equation}
where $\Gamma\subset \partial\Omega$ is a nonempty open subset. Note that $\Gamma=\partial\Omega$ is a possible scenario.
Let $A=(a_{ij})_{i,j=1}^N$ denote a measurable matrix function that fulfills the following conditions for all $\xi\in\mathbb{R}^N$:
\begin{equation}\label{01.01.2}
\Lambda w|\xi|^2\leq \sum_{i,j=1}^N a_{ij}\xi_i\xi_j\leq \Lambda^{-1}w |\xi|^2 \text{ a.e.\! in }\Omega, \quad a_{ij}\in C^1(\Omega) \text{ for } i,j=1,\cdots, N,
\end{equation}
where $\Lambda>0$ is a given constant.
In this paper, we focus on two types of degenerate partial differential equations  which are characterized by the same principal operator.
 The first type  is the following degenerate elliptic equation:
\begin{equation}\label{01.01.E1}
\begin{cases}
-\Div(A\nabla \varphi)=f &\text{in } \Omega,\\
\varphi=0 &\text{on }\partial\Omega,
\end{cases}
\end{equation}
where $f\in L^2(\Omega;w^{-1})$ is a given function. Here, $\nabla \varphi=(\frac{\partial \varphi}{\partial x_1},\cdots, \frac{\partial \varphi}{\partial x_N})$ represents the gradient of $\varphi$, $\Div F=\sum_{i=1}^N\frac{\partial F_i}{\partial x_i}$ denotes the divergence of $F=(F_1,\cdots, F_N)$, and $L^2(\Omega;w^{-1})$ is a weighted Sobolev space that will be defined subsequently.

The second  type  is the following  degenerate parabolic equation:
\begin{equation}\label{01.01.E2}
\begin{cases}
\partial_t\varphi-\Div(A\nabla \varphi)=f &\text{in }Q, \\
\varphi=0 &\text{on }\partial Q, \\
\varphi(0)=\varphi_0 &\text{in }\Omega,
\end{cases}
\end{equation}
where {{$Q=\Om\ts (0,T), \pt Q=\pt\Om\ts (0,T)$}}, and $f\in L^2(Q;w^{-1})$ is a given function (to be specified later), and $\varphi_0\in L^2(\Omega)$.
The solutions to the degenerate elliptic equation \eqref{01.01.E1} and the degenerate parabolic equation \eqref{01.01.E2} have been investigated in \cite{CS2,CS3,Fabes,GC,Heinonen,Pingen,Stuart,Trudinger}. The approximation of these equations by uniformly elliptic partial differential equations has been explored in \cite{Cavalheiro1,Cavalheiro,Cora,FF,Wu}.
Shape design problems have been addressed in numerous papers, including \cite{Buttazzo,Chenais,Greco,Guo,Guo1,He,Mu1,Privat,Tiba}. The unique continuation property and controllability for non-degenerate or degenerate {{partial differential equations have been examined in \cite{Alabau,Cannarsa1,Cannarsa2,Cannarsa3,FG,Vessella,Wu,Wu1}}}.
In this study, we introduce a novel approximation method for the equations \eqref{01.01.E1} and \eqref{01.01.E2}, termed shape design approximation. To the best of our knowledge, no prior paper has employed this approximation method to obtain solutions to \eqref{01.01.E1} and \eqref{01.01.E2}.   An exception could be the work by \cite{Greco}, published in Italian, which employed an approximation method for the degenerate case. However, this resource is, unfortunately, inaccessible to us.  Towards the end of this paper, we present an application of this approximation method, namely, a Carleman estimate for the backward degenerate parabolic equation, from which we can deduce the null controllability for the degenerate parabolic equation \eqref{01.01.E2} (refer to \cite{Alabau}).

The paper is organized  as follows. In Section \ref{S2}, we provide definitions for the solution spaces corresponding to the equations \eqref{01.01.E1} and \eqref{01.01.E2}. Section \ref{S3} is devoted  to an in-depth discussion of the shape design approximation method. Subsequently, in Section \ref{S4}, we derive the Carleman estimate for the backward one-dimensional degenerate equation. Finally, Section \ref{Se5} presents concluding remarks that highlight the novelty of our proposed methodology.

 \section{Solution spaces}\label{S2}
In this section, we shall present the solution spaces for equations \eqref{01.01.E1} and \eqref{01.01.E2}, which are classical weighted Sobolev spaces. We assume that $p\in (1,\infty)$.

\subsection{Degenerate elliptic case}
We define the weighted $L^p$ space as follows:
\begin{equation*}
L^p(\Omega; w)=\left\{u\in \mathcal{D}'(\Omega)\colon \int_\Omega |u|^p w\,\mathrm{d}x<\infty\right\},
\end{equation*}
where $\mathcal{D}'(\Omega)$ denotes the space of real distributions on $\Omega$ (refer to \cite{Fabes,GC,Heinonen,Trudinger})
and $w$ is  an $A_p$-weight function. For $L^2(\Omega;w)$, we can define an inner product:
\begin{equation*}
(u,v)_{L^2(\Omega; w)}=\int_\Omega uv w\,\mathrm{d}x.
\end{equation*}
It is well established that  $(L^2(\Om; w), (\cdot,\cdot)_{L^2(\Om;w)})$ forms a Hilbert space, and
 $(L^p(\Om;w),\|\cdot\|_{L^p(\Om;w)})$  is a Banach space with the norm defined as
\begin{equation*}
\|u\|_{L^p(\Omega;w)}=\left(\int_\Omega |u|^p w\,\mathrm{d}x\right)^{\frac{1}{p}}.
\end{equation*}
Furthermore, we note that
\begin{equation*}
u\in L^p(\Omega;w^{-1}) \text{ if and only if } \int_\Omega |u|^p w^{-1}\,\mathrm{d}x<\infty.
\end{equation*}
Now, we introduce the weighted first order Sobolev space:
\begin{equation*}
\begin{split}
W^{1,p}(\Omega;w)=\left\{u\in \mathcal{D}'(\Omega)\colon \int_\Omega |u|^p w\,\mathrm{d}x+\int_\Omega |\nabla u|^p w\,\mathrm{d}x<\infty\right\}.
\end{split}
\end{equation*}
For $H^1(\Omega;w):=W^{1,2}(\Omega;w)$, we can define an inner product:
\begin{equation*}
(u,v)_{H^1(\Omega;w)}=\int_\Omega uv w\,\mathrm{d}x+\int_\Omega (\nabla u\cdot \nabla v) w\,\mathrm{d}x.
\end{equation*}
It is a well known fact that $(H^1(\Omega;w),(\cdot, \cdot)_{H^1(\Omega;w)})$ is a Hilbert space (see \cite{Cavalheiro1, Cavalheiro,Fabes, FF,Trudinger}), and $(W^{1,p}(\Omega;w), \|\cdot \|_{W^{1,p}(\Omega;w)})$ is a Banach space, where the norm is given by
\begin{equation*}
\|u\|_{W^{1,p}(\Omega;w)}=\left(\int_\Omega |u|^p w\,\mathrm{d}x+\int_\Omega |\nabla u|^p w\,\mathrm{d}x\right)^{\frac{1}{p}}.
\end{equation*}
Additionally, we define the closure of the space of smooth functions with compact support in the $W^{1,p}(\Omega;w)$ norm:
\begin{equation*}
W_0^{1,p}(\Omega;w)=\overline{C_0^\infty(\Omega)}^{\|\cdot\|_{W^{1,p}(\Omega;w)}},
\end{equation*}
and $H_0^1(\Omega;w)=W_0^{1,2}(\Omega;w)$.




 We are now in a position to provide the definition of the solution to the  equation \eqref{01.01.E1}.
\begin{definition}\label{01.01.D1}
A function $\vp\in H_0^1(\Om;w)$ is said to be a {\it weak solution} of the degenerate elliptic equation \eqref{01.01.E1} if it satisfies the following condition:
\begin{equation*}
\int_\Om A \nabla\vp(x)\cdot \nabla \psi(x) \df x=\int_\Om f(x)\psi(x)\df x
\end{equation*}
for every $\psi\in C_0^\iy(\Om)$.
\end{definition}
\begin{lemma}\label{01.01.L2}
Let $\Om\s\R^N$ be an open and bounded domain, and let $w\in A_p$ with $1<p<\iy$. Then, there exist positive constants $C_\Om$ and $\de$ such that for all $u\in W_0^{1,p}(\Om;w)$ and all $\theta$ satisfying $1\leq \theta\leq \f{N}{N-1}+\de$, the following inequality holds:
\begin{equation}\label{01.01.3}
\|u\|_{L^{\theta p}(\Om;w)}\leq C_\Om \|\nabla u\|_{L^p(\Om;w)}.
\end{equation}
\end{lemma}
\begin{proof}
The proof of this lemma can be found in  \cite[Theorem 1.3]{Fabes}, or alternatively in (1.5) of \cite{FF}, or in   \cite[Theorem 2]{Cavalheiro}.
\end{proof}
\begin{lemma}\label{01.01.L3}
The space defined as
\begin{equation*}
H_0^1(\Om; A):=\left\{u\in\mcD'(\Om)\colon \int_\Om A\nabla u\cdot \nabla u\df x<\iy\right\}
\end{equation*}
is equivalent to the space $H_0^1(\Om;w)$.
\end{lemma}
\begin{proof}
By utilizing  condition \eqref{01.01.2} and Lemma \ref{01.01.L2}, we can establish the equivalence of these two spaces.
\end{proof}
\begin{lemma}\label{01.01.L1}
Suppose $f\in L^2(\Om;w^{-1})$. Then, the equation \eqref{01.01.E1} admits a unique solution.
\end{lemma}
\begin{proof}
This result follows directly from Lemmas \ref{01.01.L2} and \ref{01.01.L3}, in conjunction with the Lax-Milgram Theorem. For a detailed proof, one may refer to   \cite[Theorem 2.2]{Fabes} or  \cite[Theorem 2.10]{Cavalheiro1}.
\end{proof}

 \subsection{Degenerate parabolic case}
Let us denote
\begin{equation*}
L^p(Q; w)=\left\{u\in\mathcal{D}'(Q)\colon \iint_Q |u|^p w\,\mathrm{d}x\,\mathrm{d}t<\infty\right\},
\end{equation*}
and
\begin{equation*}
H^{1,p}(Q;w)=\iint_Q |u|^p w\,\mathrm{d}x\,\mathrm{d}t+\iint_Q |\nabla u|^p w\,\mathrm{d}x\,\mathrm{d}t.
\end{equation*}
Here, $H_0^{1,p}(Q;w)$ represents the closure of Lipschitz functions with compact support in $Q$ under the norm of $H^{1,p}(Q;w)$.
\begin{definition}\label{01.01.D2}
We say that $\varphi\in H_0^{1,2}(Q; w)$ is a \textit{weak solution} of the equation in \eqref{01.01.E2} if
\begin{equation*}
\begin{split}
-\iint_Q \varphi \partial_t\psi\,\mathrm{d}x\,\mathrm{d}t+\iint_Q A\nabla \varphi\cdot\nabla\psi\,\mathrm{d}x\,\mathrm{d}t=\iint_Q f\psi\,\mathrm{d}x\,\mathrm{d}t+\int_\Omega \varphi_0(x) \psi(x,0)\,\mathrm{d}x
\end{split}
\end{equation*}
holds for any $\psi\in W=\{\psi\in H_0^{1,2}(Q;w)\colon \partial_t\psi\in H^{-1,2}(Q;w)\}$ such that $\psi(T)=0$. Here, $H^{-1,2}(Q;w)$ denotes the dual space of $H_0^{1,2}(Q;w)$.
\end{definition}
\begin{lemma}\label{01.01.L4}
{{Let $\Lambda>0$ be defined in \eqref{01.01.2}.}} Let $w$ be an $A_2$ weight function, $f\in L^2(Q;w^{-1})$, and $\varphi_0\in L^2(\Omega)$. Then, there exists a weak solution of the equation \eqref{01.01.E2}, and it satisfies
\begin{equation*}
\begin{split}
&\sup_{t\in [0,T]}\int_\Omega |\varphi(x,t)|^2\,\mathrm{d}x+\iint_Q |\nabla \varphi|^2 w\,\mathrm{d}x\,\mathrm{d}t\leq C\left(\|\varphi_0\|_{L^2(\Omega)}+\|f\|_{L^2(\Omega;w^{-1})}^2\right),
\end{split}
\end{equation*}
where the positive constant $C$ depends only on $\Lambda$ and $\Omega$.
\end{lemma}
\begin{proof}
For every $\tau\in (0,T)$, by Definition \ref{01.01.D2}, we substitute $\varphi$ for $\psi$ and follow a similar argument as in the proof of Theorem 2.4 in \cite{CS3} or the proof of Lemma 3.7 in \cite{FF}. This yields
\begin{equation*}
\begin{split}
&\frac{1}{2}\int_\Omega \varphi(x,\tau)^2\,\mathrm{d}x-\frac{1}{2}\int_\Omega \varphi_0(x)^2\,\mathrm{d}x+\iint_{\Omega\times (0,\tau)} A\nabla \varphi\cdot \nabla \varphi\,\mathrm{d}x\,\mathrm{d}t=\iint_Q f\varphi\,\mathrm{d}x\,\mathrm{d}t.
\end{split}
\end{equation*}
Note that
\begin{equation*}
\iint_Q f\varphi\,\mathrm{d}x\,\mathrm{d}t\leq \left(\iint_Q f^2 w^{-1}\,\mathrm{d}x\,\mathrm{d}t\right)^{\frac{1}{2}}\left(\iint_Q \varphi^2 w\,\mathrm{d}x\,\mathrm{d}t\right)^{\frac{1}{2}}.
\end{equation*}
Then, by Lemma \ref{01.01.L2}, we obtain
\begin{equation*}
\begin{split}
\frac{1}{2}\int_\Omega \varphi(x,\tau)^2\,\mathrm{d}x+\frac{\Lambda}{2}\iint_Q |\nabla \varphi|^2 w\,\mathrm{d}x\,\mathrm{d}t\leq   \frac{1}{2}\int_\Omega \varphi_0(x)^2\,\mathrm{d}x+\frac{C_\Omega}{\Lambda}\iint_Q f^2 w^{-1}\,\mathrm{d}x\,\mathrm{d}t.
\end{split}
\end{equation*}
The existence of the solution to \eqref{01.01.E2} is established in \cite[Theorem 3.11]{FF}.
\end{proof}
 \section{Approximation}\label{S3}
In this section, we present an approximation method (Theorems \ref{01.01.T1} and \ref{01.01.T2}) for solving the equations \eqref{01.01.E1} and \eqref{01.01.E2} through shape design. This method is pivotal for establishing the Carleman estimate (Theorem \ref{01.03.T2}) in Section \ref{S4}.
We define
\begin{equation*}
\mc{O}(\Ga;\e)=\left\{x\in\Om\colon \dist(x,\Ga)>\e\right\},
\end{equation*}
where
\begin{equation*}
\dist(x,\Ga)=\inf_{y\in\Ga}|x-y|.
\end{equation*}
For every  $k\in\N$, we select a domain $\Om_k\in C^2$ such that $\mcO(\Ga,\f{1}{k})\s\Om_k\s \mc{O}(\Ga, \f{1}{k+1})$.

\subsection{Degenerate elliptic case}\label{3.S1}
Consider the equation
\begin{equation}\label{01.01.A1}
\begin{cases}
-\Div(A_k\nabla \vp_k)=f_k &\text{in } \Om_k,\\
\vp_k=0 &\text{on } \pt \Om_k,
\end{cases}
\end{equation}
where $A_k=A|_{\Om_k}$ is defined in \eqref{01.01.2}, and $f_k=f|_{\Om_k}\in L^2(\Om;w^{-1})$ is defined in \eqref{01.01.E1}. By \eqref{01.01.7}, \eqref{01.01.2}, and the definition of $\Om_k$, equation \eqref{01.01.A1} is uniformly elliptic.
\begin{definition}
A function $\vp_k\in H_0^1(\Om_k; w_k)=H_0^1(\Om_k)$ with $A_k=A|_{\Om_k}, w_k=w|_{\Om_k}$, and $f_k=f|_{\Om_k}$ is called a \textit{weak solution} to the elliptic equation \eqref{01.01.A1} if
\begin{equation*}
\int_{\Om_k} A_k\nabla\vp_k\cdot \nabla \psi\df x=\int_{\Om_k} f_k\psi\df x
\end{equation*}
for all $\psi\in C_0^\iy(\Om_k)$.
\end{definition}
\begin{lemma}\label{01.01.L5}
Define
\begin{equation*}
\be_k=\inf_{u_k\in H_0^1(\Om_k;w_k)\atop \|u_k\|_{L^2(\Om_k;w_k)}= 1}\int_{\Om_k} |\nabla u_k|^2w_k\df x.
\end{equation*}
Then, for all $k\in\N$, $\be_k\geq \Lambda C_\Om^{-2}$, where $C_\Om$ is the constant from Lemma \ref{01.01.L2}. Furthermore,
\begin{equation*}
\int_{\Om_k}u_k^2w_k\df x\leq C_\Om^2 \int_{\Om_k} |\nabla u_k|^2w_k\df x
\end{equation*}
for all $u_k\in H_0^1(\Om_k; w_k)$.
\end{lemma}
\begin{proof}
{{Since $w_k\geq C_k>0$ on $\Om_k$ by \eqref{01.01.7}}}, there exists $\vp_k\in H_0^1(\Om_k;w_k)=H_0^1(\Om_k)$ such that
\begin{equation*}
\be_k=\int_{\Om_k} A_k\nabla \vp_k\cdot \nabla\vp_k\df x=\inf_{u_k\in H_0^1(\Om_k;w_k)\atop\|u_k\|_{L^2(\Om_k;w_k)}= 1}\int_{\Om_k} A_k\nabla u_k\cdot \nabla u_k\df x.
\end{equation*}
Define
\begin{equation*}
\wh \vp_k(x)=
\begin{cases}
\vp_k(x), &x\in \Om_k,\\
0, &x\in \Om\se \Om_k.
\end{cases}
\end{equation*}
Then, $\|\wh\vp_k\|_{L^2(\Om;w)}=1$ and $\wh\vp_k\in H_0^1(\Om;w)$ for $w_k=w$ on $\Om_k$. Thus,
\begin{equation*}
\be_k=\int_{\Om_k}A_k\nabla\vp_k\cdot\nabla\vp_k\df x{{=\int_\Om A\nabla \wh\vp_k\cdot\nabla\wh \vp_k\df x} \geq \inf_{u\in H_0^1(\Om;w)\atop \|u\|_{L^2(\Om;w)}=1}\int_\Om A\nabla u\cdot \nabla u\df x=\Lambda C_\Om^{-2}}.
\end{equation*}
This completes the proof of the lemma.
\end{proof}
\begin{theorem}\label{01.01.T1}
Let $\vp_k\ (k\in\N)$ be the solution of \eqref{01.01.A1} with $A_k=A|_{\Om_k}$ and $f_k=f|_{\Om_k}$, and let $\vp$ be the solution of \eqref{01.01.E1}. Then, the sequence $\{\vp_k\}$ satisfies
\begin{equation*}
\vp_k\ura \vp \text{ weakly in } H_0^1(\Om;w) \hbox{ and } \vp_k\ura \vp\text{ weakly in } L^2(\Om;w).
\end{equation*}
Moreover,
(i) if  $H_0^1(\Om;w)\hra L^2(\Om;w)$ is compact, then
\begin{equation*}
\vp_k\ra \vp \text{ strongly in } L^2(\Om;w).
\end{equation*}
(ii) alternatively, if there exists a nonempty open subset $\om\s \Om$ such that $\ol\om\s \Om$, then
\begin{equation*}
\vp_k\ra \vp \text{ strongly in }L^2(\om;w).
\end{equation*}
\end{theorem}
\begin{proof}
Since $\vp_k$ is a solution of \eqref{01.01.A1},  {{it has}}
\begin{equation*}
\int_{\Om_k}A_k\nabla\vp_k\cdot \nabla\psi\df x=\int_{\Om_k}f_k\psi\df x
\end{equation*}
for every  $\psi\in C_0^\iy(\Om_k)$. Substituting $\psi$ with $\vp_k$ and applying the condition \eqref{01.01.2} and the Cauchy-Schwarz inequality, we obtain
\begin{equation*}
\begin{split}
\La \int_{\Om_k} |\nabla \vp_k|^2w_k \df x
&\leq \int_{\Om_k}A_k\nabla\vp_k\cdot \nabla\vp_k\df x=\int_{\Om_k}f_k\vp_k\df x\\
&\leq \left(\int_{\Om_k} f_k^2 w_k^{-1}\df x\right)^\f{1}{2}\left(\int_{\Om_k}|\vp_k|^2w_k\df x\right)^\f{1}{2},
\end{split}
\end{equation*}
where $w_k=w|_{\Om_k}$. This implies, by Lemma \ref{01.01.L5},  that
\begin{equation*}
\int_{\Om_k}|\nabla \vp_k|^2w_k\df x\leq \La^{-2}C_\Om^{2}\int_{\Om_k} f_k^2w_k^{-1}\df x.
\end{equation*}
Define
\begin{equation*}
\wh\vp_k(x)=
\begin{cases}
\vp_k(x), &x\in \Om_k,\\
0, &x\in \Om\se \Om_k.
\end{cases}
\end{equation*}
Then, $\wh\vp_k\in H_0^1(\Om)$ and
\begin{equation*}
\int_\Om |\nabla \wh\vp_k|^2w\df x=\int_{\Om_k}|\nabla \vp_k|^2 w_k\df x\leq \La^{-2}C_\Om^{2}\int_{\Om_k}f_k^2 w_k^{-1}\df x\leq  \La^{-2}C_\Om^2\int_{\Om} f^2w^{-1}\df x.
\end{equation*}
This shows that $\{\wh\vp_k\}$ is a bounded sequence in $H_0^1(\Om;w)$. Hence, there exists a subsequence of $\{\wh\vp_k\}$, {{still denoted by the same notation}}, and $\wh \vp \in H_0^1(\Om;w)$ such that
\begin{equation}\label{01.01.4}
\wh\vp_k\ura \wh \vp \text{ weakly in } H_0^1(\Om;w) \hbox{ and } \wh \vp_k\ura \wh \vp \text{ weakly in } L^2(\Om; w).
\end{equation}
Next, we show that $\wh\vp$ is a solution of \eqref{01.01.E1}.
For every $\psi\in C_0^\iy(\Om)$, by the definition of $\Om_k$, there exists $k_0\in\N$ such that for all $k\geq k_0$, $\supp \psi\s \Om_k$.
By \eqref{01.01.4} and Lemma \ref{01.01.L3}. This implies
\begin{equation*}
\begin{split}
\int_{\Om}A\nabla\wh\vp\cdot\nabla \psi\df x
&=\lim_{k\ra\iy}\int_\Om A_k\nabla\wh\vp_k \cdot \nabla\psi\df x=\lim_{k\ra\iy}\int_{\Om_k} f_k\psi\df x=\int_{\Om} f\psi\df x.
\end{split}
\end{equation*}
  This means that $\wh\vp$ is a solution of \eqref{01.01.E1}.
Finally, if we further assume that $H_0^1(\Om;w)\hra L^2(\Om;w)$ is compact, then by \eqref{01.01.4},
\begin{equation*}
\wh\vp_k\ra \wh\vp \text{ strongly in } L^2(\Om;w),
\end{equation*}
{ which is the assertion (i). }

For the assertion (ii), assume there exists a nonempty open subset $\om\s\Om$ such that $\ol\om\s \Om$. Let $\de=\f{1}{4}\dist(\om,\pt\Om)$, and choose $\zeta\in C^\iy(\R^N)$ such that $0\leq \zeta\leq 1$, and
\begin{equation*}
\zeta=1 \text{ on } \mcO(\om,\de)=\{x\in\Om\colon \dist(x,\om)< \de\},\quad \zeta=0 \text{ on }\R^N\se \mcO(\om,2\de),
\end{equation*}
with $|\nabla \zeta|\leq \f{C}{\de}$ and $|\f{\pt^2\zeta}{\pt x_i\pt x_j} |\leq \f{C}{\de^2}\ (i,j=1,\cdots, N)$ on $\R^N$, {{where the constants $C>0$ are absolute and do not depend on any other constant}}. Choose $\wh\Om\in C^2$ such that $\mcO(\om,\de)\s \wh\Om\s \ol{\wh\Om}\s \mcO(\om,3\de)\s\Om$. Then, $\zeta \wh\vp_k$ (for sufficiently large $k$) is a solution of the equation
\begin{equation*}
\begin{cases}
-\Div(A\nabla (\zeta \wh\vp_k))=\zeta f-2A\nabla \zeta \cdot\nabla \wh\vp_k-\Div(A\nabla \zeta)\wh\vp_k &\text{in }\wh\Om, \\
\zeta\wh\vp_k=0 &\text{on }\pt \wh\Om.
\end{cases}
\end{equation*}
This equation is uniformly elliptic by \eqref{01.01.2} and \eqref{01.01.7}.
By   \cite[Theorem 1, Chapter 6.3.1, p. 309]{Evans}, we have
\begin{equation*}
\begin{split}
\|\zeta\wh\vp_k\|_{H^2(\wh\Om)}
&\leq C\left(\|\zeta f-2A\nabla \zeta \cdot \nabla\wh \vp_k-\Div(A\nabla \zeta)\wh\vp_k \|_{L^2(\wh\Om)}+\|\zeta \wh\vp_k\|_{L^2(\wh\Om)}\right)\\
&\leq \f{C}{\de^2}\left(\|f\|_{L^2(\wh\Om)}+\|\wh\vp_k\|_{H^1(\wh\Om)}\right)\leq C\|f\|_{L^2(\Om;w)},
\end{split}
\end{equation*}
where the constants depend only on $\La,  w, \om$, and $\Om$. Moreover, $\|\wh\vp_k\|_{H^2(\mcO(\om,\de))}\leq C\|f\|_{L^2(\Om;w)}$. By the Rellich-Kondrachov compactness theorem in Chapter 5.7 of \cite{Evans}, we obtain
\begin{equation*}
\wh\vp_k\ra \wt \vp \text{ strongly in } \mcO(\om,\de).
\end{equation*}
Combining this with $\vp_k\ura \vp$ weakly in $L^2(\Om;w)$, we get $\wt\vp=\vp$ on $\om$. This shows that assertion  (ii) holds.
\end{proof}

 \subsection{Degenerate parabolic case}\label{3.S2}
We consider the following equation:
\begin{equation}\label{01.01.A2}
\begin{cases}
\partial_t\varphi_k - \Div(A_k\nabla \varphi_k) = f_k &\text{in } Q_k = \Omega_k \times (0,T),\\
\varphi_k = 0 &\text{on } \partial Q_k,\\
\varphi_k(0) = \varphi_0|_{\Omega_k} &\text{in } \Omega_k,
\end{cases}
\end{equation}
where $f_k = f|_{\Omega_k}$, and $\varphi_0$ is defined in \eqref{01.01.E2}.
\begin{theorem}\label{01.01.T2}
Let $\varphi$ be the solution of \eqref{01.01.E2}, and $\varphi_k\ (k\in\mathbb{N})$ be the solution of \eqref{01.01.A2} with $f_k = f|_{\Omega_k}$ and $\varphi_k(0) = \varphi_0|_{\Omega_k}$. Then the sequence $\{\varphi_k\}$ satisfies
\begin{equation*}
\varphi_k \rightharpoonup \varphi \text{ weakly in } H_0^{1,2}(Q;w) \hbox{ and } \varphi_k \rightharpoonup \varphi \text{ weakly in } L^2(Q; w).
\end{equation*}
Furthermore, if there exists  a non-empty open subset $\omega \subset \Omega$ such that $\overline{\omega} \subset \Omega$, and $\varphi_0 \in H^1(\Omega)$, then,
\begin{equation*}
\varphi_k \to \varphi \text{ strongly in } L^2(\omega; w).
\end{equation*}
\end{theorem}
\begin{proof}
Since $\varphi_k$ is a solution of \eqref{01.01.A2} for every  $k\in\mathbb{N}$, by Lemma \ref{01.01.L2}, for any $\tau\in (0,T)$, we obtain
\begin{equation*}
\begin{split}
&\frac{1}{2}\int_{\Omega_k}\varphi_k(x,\tau)^2\,dx + \Lambda\iint_{Q_k}|\nabla\varphi_k|^2w_k\,dx\,dt\\
&\leq\frac{1}{2}\int_{\Omega_k}\varphi_k(x,\tau)^2\,dx + \iint_{Q_k}A_k\nabla\varphi_k\cdot \nabla \varphi_k\,dx\,dt = \frac{1}{2}\int_{\Omega_k}\varphi_0^2\,dx + \iint_{Q_k} f_k\varphi_k\,dx\,dt\\
&\leq \frac{1}{2}\int_{\Omega_k}\varphi_0^2\,dx + \frac{1}{2\varepsilon} \iint_{Q_k} f_k^2w_k^{-1}\,dx\,dt + \varepsilon\iint_{Q_k} \varphi_k^2w_k\,dx\,dt\\
&\leq \frac{1}{2}\int_{\Omega_k}\varphi_0^2\,dx + \frac{1}{2\varepsilon} \iint_{Q_k} f_k^2w_k^{-1}\,dx\,dt + \varepsilon C_\Omega^2\iint_{Q_k} |\nabla \varphi_k|^2w_k\,dx\,dt,
\end{split}
\end{equation*}
where $A_k = A|_{\Omega_k}$, $f_k = f|_{Q_k}$, and $w_k = w|_{\Omega_k}$. Choosing $\varepsilon = \frac{\Lambda C_\Omega^{-2}}{2}$, we get
\begin{equation}\label{01.01.6}
\begin{split}
&\frac{1}{2}\int_{\Omega_k}\varphi_k(x,\tau)^2\,dx + \frac{\Lambda}{2}\iint_{\Omega_k \times (0,\tau)}|\nabla\varphi_k|^2w_k\,dx\,dt\\
&\leq \frac{1}{2}\int_{\Omega_k}\varphi_0^2\,dx + \Lambda^{-1} C_\Omega^2\iint_{\Omega_k \times (0,\tau)}f_k^2w_k^{-1}\,dx\,dt \leq \frac{1}{2}\int_\Omega \varphi_0^2\,dx + \Lambda^{-1} C_\Omega^2\iint_Q f^2w^{-1}\,dx\,dt
\end{split}
\end{equation}
for all $\tau\in (0,T)$.
Moreover, we have
\begin{equation*}
\iint_{\Omega_k \times (0,T)}\varphi_k (x,t)^2\,dx\,dt \leq T\left(\int_\Omega\varphi_0^2\,dx + 2\Lambda^{-1} C_\Omega^2\iint_Q f^2w^{-1}\,dx\,dt\right).
\end{equation*}
Define
\begin{equation*}
\widetilde{\varphi}_k(x,t)=
\begin{cases}
\varphi_k(x,t), &(x,t)\in Q_k,\\
0, &(x,t)\in Q\setminus Q_k,
\end{cases}
\end{equation*}
Then, $\widetilde{\varphi}_k\in H_0^{1,2}(Q; w)$, and hence there exists a subsequence of $\{\widetilde{\varphi}_k\}$, {{still denoted by the same notation}}, and $\widetilde{\varphi}\in H_0^{1,2}(Q; w)$ such that
\begin{equation}\label{01.03.8}
\widetilde{\varphi}_k \rightharpoonup \widetilde{\varphi} \text{ weakly in } H_0^{1,2}(Q; w) \hbox{ and } \widetilde{\varphi}_k \rightharpoonup \widetilde{\varphi} \text{ weakly in } L^2(Q; w)\cap L^2(Q).
\end{equation}

Next, we show that $\widetilde{\varphi}$ is a solution of \eqref{01.01.E2}, i.e., $\widetilde{\varphi}=\varphi$ on $Q$.

Suppose $\psi\in W$, $\psi(T)=0$. By the density argument (see the proof  of \cite[Theorem 2.3]{CS3}), we can assume $\psi\in C^\infty(\overline{Q})$ and $\psi(\cdot, t)$ is compactly supported in $\Omega$ for any $t\in [0,T]$. By the definition of $\Omega_k$, $k\in\mathbb{N}$, there exists $k_0\in\mathbb{N}$ such that $\supp \psi \subset Q_k$ for all $k\geq k_0$. Note that we have
\begin{equation*}
\begin{split}
-\iint_Q \widetilde{\varphi} \psi_t\,dx\,dt + \iint_Q A\nabla \widetilde{\varphi} \cdot \nabla \psi\,dx\,dt
&=
\lim_{k\to\infty}\left(-\iint_{Q}\widetilde{\varphi}_k\psi_t\,dx\,dt + \iint_{Q}A_k\nabla\widetilde{\varphi}_k \cdot \nabla\psi\,dx\,dt\right)\\
&=\lim_{k\to\infty}\left(\iint_{Q_k}f_k\psi\,dx\,dt + \int_{\Omega_k}\varphi_0(x)\psi(x,0)\,dx\right)\\
&=\iint_Q f\psi\,dx\,dt + \int_\Omega \varphi_0(x)\psi(x,0)\,dx
\end{split}
\end{equation*}
by \eqref{01.03.8}.
This implies that $\widetilde{\varphi}\in H_0^{1,2}(Q;w)$ is a solution of \eqref{01.01.E2}.

Finally, since $\overline{\omega} \subset \Omega$ is compact, we denote $\delta = \frac{1}{4}\dist(\overline{\omega},\partial\Omega)$. Choosing $\zeta\in C^\infty(\mathbb{R}^N)$, $0\leq \zeta\leq 1$ satisfying
\begin{equation*}
\zeta = 1 \text{ on } \mathcal{O}(\omega,\delta)=\{x\in\Omega\colon \dist(x,\omega)< \delta\},\quad \zeta = 0 \text{ on } \Omega\setminus \mathcal{O}(\omega,2\delta),
\end{equation*}
and
\begin{equation*}
|\nabla\zeta|\leq \frac{C}{\delta} \text{ and } \left|\frac{\partial^2\zeta}{\partial x_i\partial x_j}\right|\leq \frac{C}{\delta^2} \text{ on } \mathbb{R}^N,
\end{equation*}
where the constants $C>0$ are absolute constants. Denote $\widehat{\Omega}\in C^2$ satisfying $\supp \zeta \subset \widehat{\Omega} \subset \overline{\widehat{\Omega}} \subset \mathcal{O}(\omega,3\delta) \subset \Omega$.

By the definition of $\Omega_k$, there exists $k_0\in\mathbb{N}$ such that for all $k\geq k_0$, it has  $\widehat{\Omega} \subset \Omega_k$. Hence $\zeta\varphi_k$ is a solution of the following equation
\begin{equation}\label{01.01.5}
\begin{cases}
\partial_t(\zeta\varphi_k) - \Div(A\nabla(\zeta\varphi_k))
= \zeta f - 2A \nabla\zeta\cdot \nabla\varphi_k - \varphi_k\Div(A\nabla \zeta), &\text{in } \widehat{\Omega} \times (0,T),\\
\zeta\varphi_k = 0, &\text{on } \partial \widehat{\Omega} \times (0,T),\\
\zeta\varphi_k(0) = \zeta\varphi_0, &\text{in } \widehat{\Omega}.
\end{cases}
\end{equation}
Note that $A\in C^1(\overline{\widehat{\Omega}})$. Then,  $\Div(A\nabla\zeta)\in C(\overline{\widehat{\Omega}})$, and hence \eqref{01.01.5} is a uniformly parabolic equation.
Multiplying $\zeta\partial_t\varphi_k$ on both sides of \eqref{01.01.5}, and integrating on $\widehat{\Omega} \times (0,\tau)$, $\tau\in (0,T)$, we have
\begin{equation*}
\begin{split}
&\iint_{\widehat{\Omega} \times (0,\tau)}|\partial_t(\zeta\varphi_k)|^2\,dx\,dt + \frac{1}{2}\int_{\widehat{\Omega}}A\nabla (\zeta \varphi_k(\tau))\cdot \nabla (\zeta \varphi_k(\tau))\,dx\\
&=\frac{1}{2}\int_{\widehat{\Omega}} A\nabla(\zeta \varphi_0)\cdot \nabla (\zeta \varphi_0)\,dx + \iint_{\widehat{\Omega}\times (0,\tau)} \zeta^2f \partial_t\varphi_k\,dx\,dt\\
&\hspace{4.5mm}-2\iint_{\widehat{\Omega}\times (0,\tau)}(A\nabla\zeta \cdot \nabla \varphi_k)(\zeta\partial_t\varphi_k)\,dx\,dt - \iint_{\widehat{\Omega}\times (0,T)}\varphi_k \Div(A\nabla \zeta)(\zeta\partial_t\varphi_k)\,dx\,dt\\
&\leq \frac{1}{2}\int_\Omega A\nabla (\zeta \varphi_0)\cdot\nabla (\zeta \varphi_0)\,dx + \frac{1}{2\varepsilon}\iint_{\widehat{\Omega} \times (0,T)} \zeta^2f^2w_k^{-1}\,dx\,dt + \varepsilon\iint_{\widehat{\Omega}\times (0,T)}\zeta^2(\partial_t \varphi_k)^2w_k\,dx\,dt\\
&\hspace{4.5mm}+\frac{1}{\varepsilon}\iint_{\widehat{\Omega}\times (0,T)}|A\nabla\zeta\cdot \nabla\varphi_k|^2 \,dx\,dt + \varepsilon\iint_{\widehat{\Omega}\times (0,T)} \zeta^2 (\partial_t  \varphi_k)^2\,dx\,dt + \frac{1}{2\varepsilon}\iint_{\widehat{\Omega}\times (0,T)}|\Div(A\nabla\zeta)|^2 \varphi_k^2\,dx\,dt\\
&\leq \frac{C}{\delta^2}\|\varphi_0\|_{H^1(\Omega)}^2 + C\iint_Q f^2w^{-1}\,dx\,dt + C\iint_{Q_k} |\nabla\widetilde{\varphi}_k|^2w_k\,dx\,dt\\
&\leq C\left(\|\varphi_0\|_{H^1(\Omega)}^2 + \iint_Q f^2w^{-1}\,dx\,dt\right)
\end{split}
\end{equation*}
by taking $\varepsilon = \min\{\frac{1}{4\gamma},\frac{1}{4}\}$ with $\gamma = \inf_{x\in \widehat{\Omega}}w(x)$ and Lemma \ref{01.01.L2}, and the constants $C>0$ depend only on $\delta,\Lambda, C_\Omega, \gamma$ and $\Omega$.
This implies that $\{\zeta\varphi_k\}$ is a bounded sequence in
\begin{equation*}
\widehat{W}:=\left\{u\in L^2(\widehat{\Omega} \times (0,T))\colon \partial_t u\in L^2(\widehat{\Omega} \times (0,T)), u\in L^2(0,T; H_0^1(\widehat{\Omega})\cap H^2(\widehat{\Omega}))\right\}.
\end{equation*}
Since $\widehat{W} \hookrightarrow L^2(\widehat{\Omega} \times (0,T))$ is compact, there exists a subsequence of $\{\varphi_k\}$, {{still denoted by the same notation}}, and $\widetilde{\vp}\in \widehat{W}$ such that
\begin{equation*}
\zeta \varphi_k \to \widetilde{\vp} \text{ strongly in } L^2(\widehat{\Omega} \times (0,T)).
\end{equation*}
Noting that  $\widetilde{\varphi}_k \rightharpoonup \varphi$ weakly in $L^2(Q;w)\cap L^2(Q)$ and $\zeta\varphi_k = \varphi_k = \widetilde{\varphi}_k$ on $\mathcal{O}(\omega,\delta) \times (0,T) \subset Q$, we get $\widetilde{\vp} = \varphi = \widetilde{\varphi}$ on $\mathcal{O}(\omega,\delta)$.
 This completes the proof of the theorem.
\end{proof}

 \section{Application to 1-d degenerate parabolic equation}\label{S4}
In this section, we aim to derive the Carleman estimate (Theorem \ref{01.03.T2}) for the backward degenerate parabolic equation \eqref{01.02.1}. The weight function \eqref{05.06.1} employed for this estimate differs slightly from the weight (3.2) presented in \cite[p.175]{Alabau}. By utilizing this weight function, we can circumvent the need for boundary estimates. To establish the Carleman estimate, we partition the estimation process into three distinct parts. By integrating these three estimates, we ultimately obtain the desired Carleman estimate. Finally, we employ an approximation method to prove Theorem \ref{01.03.T2}. The advantage of this approximation method lies in its independence from the second partial derivatives of equation \eqref{01.02.1}. Given the challenges in defining second derivatives for equations \eqref{01.01.E1} and \eqref{01.01.E2}, particularly in higher-dimensional degenerate partial differential equations, this method proves particularly advantageous.
We focus on the case where $N=1$ and the weight function is $w=x^\alpha$, with $\alpha\in (0,1)$. The investigation of the case where $N=1$ and $w=x^\alpha$, with $\alpha\in (1,2)$ (or $N\geq 2$) will be addressed in a separate study.
We consider the following degenerate parabolic equation:
\begin{equation}\label{01.02.1}
\begin{cases}
\partial_t \varphi+\partial_x(w \partial_x \varphi)=g & \text{in } Q, \\
\varphi=0 &\text{on }\partial Q, \\
\varphi(T)=\varphi_T &\text{in }\Omega,
\end{cases}
\end{equation}
where $Q=\Omega\times(0,T)$, $\Omega=(0,1)$, $w=x^\alpha$, $\alpha\in (0,1)$, $g\in L^2(\Omega;w^{-1})$, and $\varphi_T\in L^2(\Omega)$. We observe that, from
\begin{equation*}
\iint_Q g^2\,dx\,dt=\iint_Q g^2w^{-1}w\,dx\,dt\leq \iint_Q g^2w^{-1}\,dx\,dt,
\end{equation*}
it follows that $g\in L^2(Q)$.
To derive the Carleman estimate (Theorem \ref{01.03.T2}) for equation \eqref{01.02.1}, we rely on Theorem \ref{01.01.T2}. Consequently, we first need to establish the Carleman estimate (Theorem \ref{01.03.T1}) for the following equation \eqref{01.02.2}.
Let $\varphi_k$ ($k\in\mathbb{N}$) denote the solution of the equation:
\begin{equation}\label{01.02.2}
\begin{cases}
\partial_t\varphi_k+\partial_x(w\partial_x\varphi_k)=g_k &\text{in }Q_k,\\
\varphi_k=0 &\text{on } \partial Q_k,\\
\varphi_k(T)=\varphi_T|_{\Omega_k} &\text{in }\Omega_k,
\end{cases}
\end{equation}
where $\Omega_k=(\frac{1}{k}, 1)$, $Q_k=\Omega_k\times (0,T)$, and $g_k=g|_{\Omega_k}$. We note that \eqref{01.02.2} represents a uniformly parabolic equation for any $k\in\mathbb{N}$.
We assume $\omega=(a,b)\subset \Omega$ with $0<a<b<1$. We select $\eta\in C(\overline{\Omega})\cap C^3(\Omega)$ satisfying
\begin{equation}\label{05.06.1}
\eta(x)=
\begin{cases}
x^{2-\alpha} &\text{for } x\in (0, \frac{2a+b}{3}),\\
>0 &\text{for }x\in (\frac{2a+b}{3}, \frac{a+2b}{3}), \\
(1-x)x^{-\alpha} &\text{for }x\in (\frac{a+2b}{3}, 1).
\end{cases}
\end{equation}
We define
\begin{equation}\label{gbz10}
\Theta=\frac{1}{[t(T-t)]^4},\quad \xi=\gamma\Theta[-2|\eta|_\infty+\eta],
\end{equation}
where  $|\eta|_\infty=\sup_{x\in \Omega}|\eta(x)|$ and $\gamma>0$  is a constant to be specified later. It is straightforward to verify that
\begin{equation*}
\partial_x\xi=\gamma \Theta (\partial_x\eta).
\end{equation*}
Let
\begin{equation}\label{01.03.4}
v_k=e^{s\xi}\varphi_k\ \Leftrightarrow{}\ \varphi_k=e^{-s\xi}v_k.
\end{equation}
Then, we have the following properties (noting that $\xi$ is independent of $k\in\mathbb{N}$):

i) $v_k=\partial_xv_k=0$ at $t=0$ and $t=T$;

ii) $v_k=0$ on $\partial Q_k$.

By direct computation, we obtain
\begin{equation*}
e^{s\xi}g=P_1v_k+P_2v_k,
\end{equation*}
where
\begin{equation}\label{gbz1}
\begin{split}
P_1v_k
&=\sum_{i=1}^3P_{1i}v_k=\partial_t v_k-2sw (\partial_x v_k)(\partial_x\xi)-sv_k\partial_x(w\partial_x\xi),\\
P_2v_k
&=\sum_{i=1}^3P_{2i}v_k=\partial_x(w\partial_xv_k)-sv_k\partial_t\xi+s^2v_kw(\partial_x\xi)^2.
\end{split}
\end{equation}

 We now proceed to compute $(P_1v_k,P_2v_k)_{L^2(Q_k)}$ term-by-term as presented in \dref{gbz1}. It is important to note that \eqref{01.02.2} represents a uniform parabolic equation, thereby ensuring the meaningfulness of the second derivatives involved.
 This is split into several steps.

{\it Step 1: Computation of $(P_{11}v_k,P_{21}v_k)_{L^2(Q_k)}$.}

Given i) and the condition $\pt_tv_k=0$ on $\pt Q_k$, we derive:
\begin{equation}\label{gbz2}
\begin{split}
&(P_{11}v_k,P_{21}v_k)_{L^2(Q_k)}\\
&=\iint_{Q_k}(\pt_tv_k)\pt_x(w\pt_xv_k)\,dx\,dt=-\iint_{Q_k}w(\pt_xv_k)\pt_x(\pt_tv_k)\,dx\,dt\\
&=-\frac{1}{2}\iint_{Q_k}\pt_t\left[w(\pt_xv_k)^2\right]\,dx\,dt=0.
\end{split}
\end{equation}

{\it Step 2:  Computation of $(P_{12}v_k,P_{21}v_k)_{L^2(Q_k)}$.}

We calculate as follows:
\begin{equation}\label{gbz3}
\begin{split}
&(P_{12}v_k,P_{21}v_k)_{L^2(Q_k)}\\
&=-2s\iint_{Q_k} w(\pt_xv_k)(\pt_x\xi)\pt_x(w\pt_xv_k)\,dx\,dt\\
&={{2s\int_0^T w^2(k^{-1})(\pt_x\xi)(k^{-1},t)(\pt_xv_k)^2(k^{-1},t)\,dt-2s\int_0^T w^2(1)(\pt_x\xi)(1,t)(\pt_xv_k)^2(1,t)\,dt}}\\
&\quad+2s\iint_{Q_k} w(\pt_xv_k)\pt_x\big[w(\pt_xv_k)(\pt_x\xi)\big]\,dx\,dt\\
&={{s\int_0^T w^2(k^{-1})(\pt_x\xi)(k^{-1},t)(\pt_xv_k)^2(k^{-1},t)\,dt-s\int_0^T w^2(1)(\pt_x\xi)(1,t)(\pt_xv_k)^2(1,t)\,dt}}\\
&\quad+s\iint_{Q_k} w^2(\pt_xv_k)^2(\pt_x^2\xi)\,dx\,dt.
\end{split}
\end{equation}

{\it Step 3: Computation of $(P_{13}v_k,P_{21}v_k)_{L^2(Q_k)}$.}

Using ii), we obtain:
\begin{equation}\label{gbz4}
\begin{split}
(P_{13}v_k,P_{21}v_k)_{L^2(Q_k)}
&=-s\iint_{Q_k} v_k\big[\pt_x(w\pt_x\xi)\big]\pt_x(w\pt_xv_k)\,dx\,dt\\
&=s\iint_{Q_k}w(\pt_xv_k)^2\pt_x(w\pt_x\xi)\,dx\,dt+s\iint_{Q_k}v_kw(\pt_xv_k)\pt_x^2(w\pt_x\xi)\,dx\,dt.
\end{split}
\end{equation}

{\it Step 4: Computation of $(P_1v_k,P_{22}v_k)_{L^2(Q_k)}$.}

Combining i) and ii), we have:
\begin{equation}\label{gbz5}
\begin{split}
&(P_1v_k,P_{22}v_k)_{L^2(Q_k)}\\
&=\frac{s}{2}\iint_{Q_k}v_k^2\pt_{t}^2\xi\,dx\,dt-s^2\iint_{Q_k}v_k^2\pt_x\big[w(\pt_t\xi)(\pt_x\xi)\big]\,dx\,dt+s^2\iint_{Q_k}v_k^2(\pt_t\xi)\pt_x(w\pt_x\xi)\,dx\,dt\\
&=\frac{s}{2}\iint_{Q_k}v_k^2\pt_t^2\xi\,dx\,dt-s^2\iint_{Q_k}v_k^2w(\pt_x\xi)\pt_{xt}\xi\,dx\,dt.
\end{split}
\end{equation}

{\it Step 5: Computation of $(P_{11}v_k, P_{23}v_k)_{L^2(Q_k)}$.}

From i), we derive
\begin{equation}\label{gbz6}
\begin{split}
(P_{11}v_k, P_{23}v_k)_{L^2(Q_k)}
&=s^2\iint_{Q_k}v_kw(\pt_x\xi)^2 \pt_tv_k\,dx\,dt=\frac{s^2}{2}\iint_{Q_k}w(\pt_x\xi)^2\pt_tv_k^2\,dx\,dt\\
&=-s^2\iint_{Q_k} v_k^2 w(\pt_x\xi)\pt_{xt}\xi\,dx\,dt.
\end{split}
\end{equation}

{\it Step 6: Computation of $(P_{12}v_k+P_{13}v_k, P_{23}v_k)_{L^2(Q_k)}$.}

We obtain
\begin{equation}\label{gbz7}
\begin{split}
&(P_{12}v_k+P_{13}v_k, P_{23}v_k)_{L^2(Q_k)}\\
&=-2s^3\iint_{Q_k}v_kw^2(\pt_xv_k)(\pt_x\xi)^3\,dx\,dt-s^3\iint_{Q_k}v_k^2 w(\pt_x\xi)^2\pt_x(w\pt_x\xi)\,dx\,dt\\
&=s^3\iint_{Q_k}v_k^2w(\pt_x\xi)\pt_x\big[w(\pt_x\xi)^2\big]\,dx\,dt.
\end{split}
\end{equation}

 \subsection{Estimation}
Let us define the following regions:
\begin{equation*}
Q_k^1=\left(0,\frac{2a + b}{3}\right) \times (0,T), \quad Q_k^3=\left(\frac{a + 2b}{3},1\right) \times (0,T), \quad Q_k^2 = Q_k \setminus (Q_k^1 \cup Q_k^3).
\end{equation*}

The estimation is split into several steps.

{\it Step 1.} It can be readily verified that the first order partial derivative of $\xi$  defined by \dref{gbz10} with respect to $x$ is given by
\begin{equation*}
\partial_x\xi=
\begin{cases}
(2 - \alpha)\gamma\Theta x^{1 - \alpha}, &x\in Q_k^1,\\
\gamma\Theta\left[-x^{-\alpha}+(-\alpha)(1 - x)x^{-\alpha - 1}\right], &x\in Q_k^3,
\end{cases}
\end{equation*}
the second  order partial derivative of $\xi$ with respect to $x$ is
\begin{equation*}
\partial_x^2\xi=
\begin{cases}
(2 - \alpha)(1 - \alpha)\gamma \Theta x^{-\alpha}, &x\in Q_k^1,\\
\gamma\Theta\left[2\alpha x^{-\alpha - 1}+\alpha (\alpha + 1)(1 - x)x^{-\alpha - 2}\right], &x\in Q_k^3,
\end{cases}
\end{equation*}
the partial derivative of $w\partial_x\xi$ with respect to $x$ is
\begin{equation*}
\partial_x(w\partial_x\xi)=
\begin{cases}
(2 - \alpha)\gamma\Theta, &x\in Q_k^1,\\
\gamma\Theta\left[\alpha x^{-1}+\alpha(1 - x)x^{-2}\right], &x\in Q_k^3,
\end{cases}
\end{equation*}
and the second  order partial derivative of $w\partial_x\xi$ with respect to $x$ is
\begin{equation*}
\partial_x^2(w\partial_x\xi)=
\begin{cases}
0, &x\in Q_k^1,\\
\gamma\Theta\left[-2\alpha x^{-3}\right], &x\in Q_k^3.
\end{cases}
\end{equation*}
Combining these results with $\partial_x\xi(1,t)=-\gamma\Theta$, we obtain
\begin{equation*}
\begin{split}
E_1(Q_k)
&=(P_{1}v_k,P_{21}v_k)_{L^2(Q_k)}\\
&\geq (2 - \alpha)^2s\gamma \iint_{Q_k^1}\Theta x^{\alpha}(\partial_xv_k)^2\mathrm{d}x\mathrm{d}t + Cs\gamma\iint_{Q_k^3} \Theta (\partial_x v_k)^2\mathrm{d}x\mathrm{d}t\\
&\quad - Cs\gamma\iint_{Q_k^2} \Theta (\partial_xv_k)^2\mathrm{d}x\mathrm{d}t - Cs\gamma\iint_{Q_k^2\cup Q_k^3}\Theta |v_k\partial_xv_k|\mathrm{d}x\mathrm{d}t\\
&\geq (2 - \alpha)^2s\gamma \iint_{Q_k^1}\Theta x^{\alpha}(\partial_xv_k)^2\mathrm{d}x\mathrm{d}t + Cs\gamma\iint_{Q_k^3} \Theta (\partial_x v_k)^2\mathrm{d}x\mathrm{d}t\\
&\quad - Cs\gamma\iint_{Q_k^2} \Theta (\partial_xv_k)^2\mathrm{d}x\mathrm{d}t - Cs\iint_{Q_k^2\cup  Q_k^3}\Theta (\partial_xv_k)^2\mathrm{d}x\mathrm{d}t - Cs\gamma^2\iint_{Q_k^2\cup Q_k^3}\Theta v_k^2\mathrm{d}x\mathrm{d}t,
\end{split}
\end{equation*}
where the positive constants $C$ depend solely on $\alpha, a$, and $b$. { Here it is noted that the term with the coefficient $(2-\alpha)^2$ originates from two components: the final term in equation \dref{gbz3} and the initial term in the last equality of equation
\dref{gbz4}. Specifically,
\begin{equation*}
	s\iint_{Q_k^1} w^2(\partial_xv_k)^2(\partial_x^2\xi)\,dx\,dt = (2-\alpha)(1-\alpha)s\gamma \iint_{Q_k^1} \Theta x^\alpha (\partial_xv_k)^2\,dx\,dt,
\end{equation*}
and
\begin{equation*}
	s\iint_{Q_k}w(\partial_xv_k)^2\partial_x(w\partial_x\xi)\,dx\,dt = (2-\alpha)s\gamma\iint_{Q_k^1}\Theta x^\alpha (\partial_xv_k)^2\,dx\,dt.
\end{equation*}}

{\it Step 2.}  Given that
\begin{equation*}
\partial_t\xi=\gamma\Theta' \left[-2|\eta|_{\infty}+\eta\right],\quad \partial_t^2\xi=\gamma\Theta''\left[-2|\eta|_{\infty}+\eta\right],\quad |\Theta'|\leq C\Theta^{\frac{5}{4}},\quad |\Theta''|\leq C\Theta^{\frac{3}{2}},
\end{equation*}
and
\begin{equation*}
w(\partial_x\xi)(\partial_{xt}\xi)=
\begin{cases}
(2 - \alpha)^2\gamma^2\Theta\Theta' x^{2 - \alpha}, &x\in Q_k^1,\\
\leq C\gamma^2|\Theta \Theta'|, & x\in Q_k^3,
\end{cases}
\end{equation*}
we can derive
\begin{equation*}
\begin{split}
E_2(Q_k)
&=(P_1v_k,P_{22}v_k)_{L^2(Q_k)}+(P_{11}v_k, P_{23}v_k)_{L^2(Q_k)}\\
&\geq - Cs\gamma \iint_{Q_k^1}\Theta^{\frac{3}{2}}v_k^2\mathrm{d}x\mathrm{d}t - Cs^2\gamma^2\iint_{Q_k^1} \Theta^{\frac{9}{4}}v_k^2x^{2 - \alpha}\mathrm{d}x\mathrm{d}t - Cs^2\gamma^2\iint_{Q_k^2\cup Q_k^3} \Theta^{\frac{9}{4}}v_k^2 \mathrm{d}x\mathrm{d}t.
\end{split}
\end{equation*}
Here, the positive constants $C$ depend on $\alpha, a, b$, and $T$.

{\it Step 3. }  {  Given that  $\partial_x\xi=\gamma\Theta(-x^{-\alpha-1}[x+\alpha(1-x)])$ for $x\in Q_k^3$, we can obtain
$w\partial_x\xi=\gamma\Theta (-x^{-1}[x+\alpha(1-x)])$. Moreover,
\begin{equation*}
	(\partial_x\xi)^2=\gamma^2\Theta^2 x^{-2\alpha-2}[x+\alpha(1-x)]^2, \quad w(\partial_x\xi)^2=\gamma^2\Theta^2x^{-\alpha-2}[x+\alpha(1-x)]^2.
\end{equation*}
These lead to the following result: 
\begin{equation*}
	\begin{split}
	\partial_x\left[w(\partial_x\xi)^2\right]
	&=\gamma^2\Theta^2 \left((-\alpha-2)x^{-\alpha-3}[x+\alpha(1-x)]^2+2x^{-\alpha-2}[x+\alpha(1-x)](1-\alpha)\right)\\
	&=\gamma^2\Theta^2\left(-\alpha x^{-\alpha-3}[x+\alpha(1-x)][\alpha+2+x-\alpha x]\right).
	\end{split}
\end{equation*}
Consequently, for $x\in Q_k^3$, we have 
\begin{equation*}
	w(\partial_x\xi)\partial_x\left[w(\partial_x\xi)^2\right]=\gamma^3\Theta^3 \left(\alpha x^{-\alpha-4}[x+\alpha(1-x)]^2[\alpha+2+x-\alpha x]\right).
\end{equation*}
}  Now, since 
\begin{equation*}
w(\partial_x\xi)\partial_x\left[w(\partial_x\xi)^2\right]=
\begin{cases}
(2 - \alpha)^4\gamma^3\Theta^3 x^{2 - \alpha}, &x\in Q_k^1,\\
\gamma^3\Theta^3\left[\alpha x^{-4 - \alpha}(\alpha + x-\alpha x)^2(2+\alpha + x-\alpha x)\right], &x\in Q_k^3,
\end{cases}
\end{equation*}
we have
\begin{equation*}
\begin{split}
E_3(Q_k)
&=(P_{12}v_k+P_{13}v_k,P_{23}v_k)_{L^2(Q_k)}\\
&\geq (2 - \alpha)^4s^3\gamma^3\iint_{Q_k^1}\Theta^3v_k^2 x^{2 - \alpha}\mathrm{d}x\mathrm{d}t + Cs^3\gamma^3\iint_{Q_k^3} \Theta^3 v_k^2\mathrm{d}x\mathrm{d}t\\
&\quad - Cs^3\gamma^3\iint_{Q_k^2}\Theta^3 v_k^2\mathrm{d}x\mathrm{d}t.
\end{split}
\end{equation*}
Here, the positive constant $C$ depends only on $\alpha, a$, and $b$.

Finally, by combining $E_1(Q_k), E_2(Q_k)$, and $E_3(Q_k)$ with the identity
\begin{equation*}
2(P_1v_k,P_2v_k)_{L^2(Q_k)}+\|P_1v_k\|_{L^2(Q_k)}^2+\|P_2v_k\|_{L^2(Q_k)}^2=\|e^{s\xi}g\|_{L^2(Q_k)}^2,
\end{equation*}
and choosing a sufficiently large $\gamma_0\geq 1$, when $\gamma\geq \gamma_0$, we obtain
\begin{equation}\label{01.03.1}
\begin{split}
&2(2 - \alpha)^2s\gamma\iint_{Q_k^1}\Theta x^\alpha (\partial_xv_k)^2\mathrm{d}x\mathrm{d}t + Cs\gamma\iint_{Q_k^3}\Theta (\partial_xv_k)^2\mathrm{d}x\mathrm{d}t\\
&+ 2(2 - \alpha)^4s^3\gamma^3\iint_{Q_k^1} \Theta^3 v_k^2x^{2 - \alpha}\mathrm{d}x\mathrm{d}t + Cs^3\gamma^3\iint_{Q_k^3}\Theta^3 v_k^2\mathrm{d}x\mathrm{d}t+\|P_1v_k\|_{L^2(Q_k)}^2+\|P_2v_k\|_{L^2(Q_k)}^2\\
&\leq \|e^{s\xi}g_k\|_{L^2(Q_k)}^2 + Cs\gamma\iint_{Q_k^2}\Theta (\partial_xv_k)^2\mathrm{d}x\mathrm{d}t + Cs^3\gamma^3\iint_{Q_k^2}\Theta^3v_k^2\mathrm{d}x\mathrm{d}t + Cs\gamma\iint_{Q_k^1}\Theta^{\frac{3}{2}}v_k^2\mathrm{d}x\mathrm{d}t\\
&\leq \|e^{s\xi}g_k\|_{L^2(Q_k)}^2 + Cs\gamma\iint_{Q_k^2}\Theta (\partial_xv_k)^2\mathrm{d}x\mathrm{d}t + Cs^3\gamma^3\iint_{Q_k^2}\Theta^3v_k^2\mathrm{d}x\mathrm{d}t
\end{split}
\end{equation}
by applying Lemma \ref{01.03.L1} (used in the second inequality above), and
\begin{equation}\label{01.03.2}
\begin{split}
Cs\gamma\iint_{Q_k^1}\Theta^{\frac{3}{2}}v_k^2\mathrm{d}x\mathrm{d}t
&\leq Cs\gamma^2\iint_{Q_k^1} \Theta^3v_k^2x^{2 - \alpha}\mathrm{d}x\mathrm{d}t + Cs\iint_{Q_k^1}\Theta v_k^2x^{\alpha - 2}\mathrm{d}x\mathrm{d}t\\
&\leq Cs\gamma^2\iint_{Q_k^1} \Theta^3v_k^2x^{2 - \alpha}\mathrm{d}x\mathrm{d}t + Cs\iint_{Q_k}\Theta x^\alpha  (\partial_xv_k)^2\mathrm{d}x\mathrm{d}t\\
&\leq Cs\gamma^2\iint_{Q_k^1} \Theta^3v_k^2x^{2 - \alpha}\mathrm{d}x\mathrm{d}t + Cs\iint_{Q_k^1}\Theta x^\alpha (\partial_xv_k)^2\mathrm{d}x\mathrm{d}t\\
&\quad + Cs\iint_{Q_k^2\cup Q_k^3}\Theta (\partial_xv_k)^2\mathrm{d}x\mathrm{d}t.
\end{split}
\end{equation}
Here, the positive constants $C$ depend only on $\alpha, a, b$, and $T$.

 \begin{lemma}\label{01.03.L1}
Let $\alpha \in (0,1)$, and let $v_k$ be defined as in \eqref{01.03.4}. Then, the following inequality holds:
\begin{equation*}
\iint_{Q_k} \Theta x^{\alpha - 2} v_k^2 \, \mathrm{d}x \, \mathrm{d}t \leq C \iint_{Q_k} \Theta x^\alpha (\partial_x v_k)^2 \, \mathrm{d}x \, \mathrm{d}t.
\end{equation*}
\end{lemma}
\begin{proof}
It suffices to establish the following inequality:
\begin{equation*}
\int_{\Omega_k} x^{\alpha - 2} v_k^2 \, \mathrm{d}x \leq C \int_{\Omega_k} x^\alpha (\partial_x v_k)^2 \, \mathrm{d}x.
\end{equation*}
Given that $v_k = 0$ at $x = k^{-1}$ and $x = 1$, we have
\begin{equation*}
\begin{split}
2\int_{\Omega_k} x^{\alpha - 1} v_k \partial_x v_k \, \mathrm{d}x
&= \int_{\Omega_k} x^{\alpha - 1} \partial_x v_k^2 \, \mathrm{d}x = -(\alpha - 1) \int_{\Omega_k} x^{\alpha - 2} v_k^2 \, \mathrm{d}x.
\end{split}
\end{equation*}
Consequently,
\begin{equation*}
\begin{split}
(1 - \alpha) \int_{\Omega_k} x^{\alpha - 2} v_k^2 \, \mathrm{d}x
&= 2\int_{\Omega_k} \left(x^{\frac{\alpha}{2} - 1} v_k\right) \left(x^{\frac{\alpha}{2}} \partial_x v_k\right) \, \mathrm{d}x \\
&\leq 2\left(\int_{\Omega_k} x^{\alpha - 2} v_k^2 \, \mathrm{d}x\right)^{\frac{1}{2}} \left(\int_{\Omega_k} x^\alpha (\partial_x v_k)^2 \, \mathrm{d}x\right)^{\frac{1}{2}}.
\end{split}
\end{equation*}
Thus,
\begin{equation*}
 \int_{\Omega_k} x^{\alpha - 2} v_k^2 \, \mathrm{d}x \leq \frac{4}{(1 - \alpha)^2} \int_{\Omega_k} x^\alpha (\partial_x v_k)^2 \, \mathrm{d}x.
\end{equation*}
\end{proof}

{{From the definition of $P_1 v_k$ in \dref{gbz1}}}, we obtain
\begin{equation*}
\begin{split}
&s^{-1} \iint_{Q_k} \Theta^{-1} |\partial_t v_k|^2 \, \mathrm{d}x \, \mathrm{d}t \\
&\leq C \left(+s^{-1} \|P_1 v\|_{L^2(Q_k)}^2 + s \iint_{Q_k} \Theta^{-1} w^2 (\partial_x v_k)^2 (\partial_x \xi)^2 \, \mathrm{d}x \, \mathrm{d}t + s \iint_{Q_k} \Theta^{-1} v_k^2 \left[\partial_x (w \partial_x \xi)\right]^2 \, \mathrm{d}x \, \mathrm{d}t\right) \\
&\leq C \bigg(+s^{-1} \|P_1 v\|_{L^2(Q_k)}^2 + s \iint_{Q_k^1} \Theta x^\alpha (\partial_x v_k)^2 \, \mathrm{d}x \, \mathrm{d}t + s \iint_{Q_k^1} \Theta v_k^2 \, \mathrm{d}x \, \mathrm{d}t \\
&\hspace{15mm} + s \iint_{Q_k^2 \cup Q_k^3} \Theta (\partial_x v_k)^2 \, \mathrm{d}x \, \mathrm{d}t + s \iint_{Q_k^2 \cup Q_k^3} \Theta v_k^2 \, \mathrm{d}x \, \mathrm{d}t\bigg) \\
&\leq C \bigg(+s^{-1} \|P_1 v\|_{L^2(Q_k)}^2 + s \iint_{Q_k^1} \Theta x^\alpha (\partial_x v_k)^2 \, \mathrm{d}x \, \mathrm{d}t + s \iint_{Q_k^2 \cup Q_k^3} \Theta (\partial_x v_k)^2 \, \mathrm{d}x \, \mathrm{d}t \\
&\hspace{15mm} + s \iint_{Q_k^1} \Theta^3 v_k^2 x^{2 - \alpha} \, \mathrm{d}x \, \mathrm{d}t + s \iint_{Q_k^2 \cup Q_k^3} \Theta v_k^2 \, \mathrm{d}x \, \mathrm{d}t\bigg),
\end{split}
\end{equation*}
where the third inequality employs the same argument as in \eqref{01.03.2}.
Similarly, from the definition of $P_2 v_k$ defined in \dref{gbz2}, we have
\begin{eqnarray*}
&&s^{-1} \iint_{Q_k} \Theta^{-1} |\partial_x (w \partial_x v_k)|^2 \, \mathrm{d}x \, \mathrm{d}t \\
&&\leq C \left(+s^{-1} \|P_2 v_k\|_{L^2(Q_k)}^2 + s \iint_{Q_k} \Theta^{-1} v_k^2 (\partial_t \xi)^2 \, \mathrm{d}x \, \mathrm{d}t + s^3 \iint_{Q_k} \Theta^{-1} v_k^2 \left[w (\partial_x \xi)^2\right]^2 \, \mathrm{d}x \, \mathrm{d}t\right) \\
&&\leq C \bigg(+s^{-1} \|P_2 v_k\|_{L^2(Q_k)}^2 + s \iint_{Q_k} \Theta^{\frac{3}{2}} v_k^2 \, \mathrm{d}x \, \mathrm{d}t \\
&&\hspace{15mm} {{+ s^3 \iint_{Q_k^1} \Theta^3 v_k^2 x^{4- 2\alpha} \, \mathrm{d}x \, \mathrm{d}t + s^3 \iint_{Q_k^2 \cup Q_k^3} \Theta^{3} v_k^2 \, \mathrm{d}x \, \mathrm{d}t\bigg)}} \\
&&\leq C \bigg(+s^{-1} \|P_2 v_k\|_{L^2(Q_k)}^2 + s \iint_{Q_k} \Theta^{\frac{3}{2}} v_k^2 \, \mathrm{d}x \, \mathrm{d}t \\
&&\hspace{15mm} {{+ s^3 \iint_{Q_k^1} \Theta^3 v_k^2 x^{2- \alpha} \, \mathrm{d}x \, \mathrm{d}t + s^3 \iint_{Q_k^2 \cup Q_k^3} \Theta^{3} v_k^2 \, \mathrm{d}x \, \mathrm{d}t\bigg)}} \\
&&\leq C \bigg(+s^{-1} \|P_2 v_k\|_{L^2(Q_k)}^2 + s \iint_{Q_k^1} \Theta x^\alpha (\partial_x v_k)^2 \, \mathrm{d}x \, \mathrm{d}t + s \iint_{Q_k^2 \cup Q_k^3} \Theta (\partial_x v_k)^2 \, \mathrm{d}x \, \mathrm{d}t \\
&&\hspace{15mm} {{+ s^3 \iint_{Q_k^1} \Theta^{3} v_k^2 x^{2- \alpha} \, \mathrm{d}x \, \mathrm{d}t + s^3 \iint_{Q_k^2 \cup Q_k^3} \Theta^{3} v_k^2 \, \mathrm{d}x \, \mathrm{d}t\bigg)}},
\end{eqnarray*}
again using the same argument as in \eqref{01.03.2} for the third inequality.
Combining the above results with \eqref{01.03.1}, we obtain
\begin{equation}\label{01.03.3}
\begin{split}
&s \iint_{Q_k^1} \Theta x^\alpha (\partial_x v_k)^2 \, \mathrm{d}x \, \mathrm{d}t + s \iint_{Q_k^3} \Theta (\partial_x v_k)^2 \, \mathrm{d}x \, \mathrm{d}t \\
&+ s^3 \iint_{Q_k^1} \Theta^3 v_k^2 x^{2 - \alpha} \, \mathrm{d}x \, \mathrm{d}t + s^3 \iint_{Q_k^3} \Theta^3 v_k^2 \, \mathrm{d}x \, \mathrm{d}t \\
&+ s^{-1} \iint_{Q_k} \Theta^{-1} (|\partial_t v_k|^2 + |\partial_x (w \partial_x v_k)|^2) \, \mathrm{d}x \, \mathrm{d}t \\
&\leq +C \|e^{s \xi} g_k\|_{L^2(Q_k)}^2 + C s \iint_{Q_k^2} \Theta (\partial_x v_k)^2 \, \mathrm{d}x \, \mathrm{d}t + C s^3 \iint_{Q_k^2} \Theta^3 v_k^2 \, \mathrm{d}x \, \mathrm{d}t,
\end{split}
\end{equation}
where the constants $C > 0$ depend only on $\alpha$, $a$, $b$, and $T$.
Choose $\zeta \in C^\infty(\mathbb{R})$ such that
\begin{equation*}
0 \leq \zeta \leq 1, \quad \zeta = 1 \text{ on } \left(\frac{2a + b}{3}, \frac{a + 2b}{3}\right), \quad {{\zeta = 0 \text{ on } (-\infty, a) \cup (b, +\infty)}},
\end{equation*}
and
\begin{equation*}
|\zeta'| \leq \frac{C}{b - a} \text{ on } \mathbb{R}.
\end{equation*}
Note that ($0 < \varepsilon < \frac{a^\alpha}{2}$)
\begin{eqnarray*}
&&s \iint_{Q_k^2} \Theta (\partial_x v_k)^2 \, \mathrm{d}x \, \mathrm{d}t \\
&&\leq s \iint_{Q_k} \Theta \zeta (\partial_x v_k)^2 \, \mathrm{d}x \, \mathrm{d}t \\
&&= -s \iint_{Q_k} \Theta (\partial_x \zeta) v_k \partial_x v_k \, \mathrm{d}x \, \mathrm{d}t - s \iint_{Q_k} \Theta v_k \zeta w^{-1} \partial_x (w \partial_x v_k) \, \mathrm{d}x \, \mathrm{d}t \\
&&\hspace{4.5mm} + s \iint_{Q_k} \Theta v_k \zeta (\partial_x w^{-1}) (w \partial_x v_k) \, \mathrm{d}x \, \mathrm{d}t \\
&&\leq \varepsilon \left(s^{-1} \iint_{\omega \times (0, T)} \Theta^{-1} |\partial_x (w \partial_x v_k)|^2 \, \mathrm{d}x \, \mathrm{d}t + s \iint_{\omega \times (0, T)} \Theta (\partial_x v_k)^2 \, \mathrm{d}x \, \mathrm{d}t\right) \\
&&\hspace{4.5mm} + C \frac{1}{2\varepsilon} s^3 \iint_{\omega \times (0, T)} \Theta^3 v_k^2 \, \mathrm{d}x \, \mathrm{d}t \\
&&\leq \varepsilon \left(s^{-1} \iint_{\omega \times (0, T)} \Theta^{-1} |\partial_x (w \partial_x v_k)|^2 \, \mathrm{d}x \, \mathrm{d}t + \frac{1}{a^\alpha} s \iint_{Q_k} \Theta x^\alpha (\partial_x v_k)^2 \, \mathrm{d}x \, \mathrm{d}t\right) \\
&&\hspace{4.5mm} + C \frac{1}{2\varepsilon} s^3 \iint_{\omega \times (0, T)} \Theta^3 v_k^2 \, \mathrm{d}x \, \mathrm{d}t,
\end{eqnarray*}
which implies that
\begin{equation*}
\begin{split}
&\frac{1}{2} s \iint_{Q_k^2} \Theta (\partial_x v_k)^2 \, \mathrm{d}x \, \mathrm{d}t \\
&\leq \varepsilon s^{-1} \iint_{\omega \times (0, T)} \Theta^{-1} |\partial_x (w \partial_x v_k)|^2 \, \mathrm{d}x \, \mathrm{d}t + \frac{1}{a^\alpha} \varepsilon s \iint_{Q_k^1} \Theta x^\alpha (\partial_x v_k)^2 \, \mathrm{d}x \, \mathrm{d}t \\
&\hspace{4.5mm} + \frac{1}{a^\alpha} \varepsilon s \iint_{Q_k^3} \Theta (\partial_x v_k)^2 \, \mathrm{d}x \, \mathrm{d}t + C \frac{1}{2\varepsilon} s^3 \iint_{\omega \times (0, T)} \Theta^3 v_k^2 \, \mathrm{d}x \, \mathrm{d}t,
\end{split}
\end{equation*}
where the constants $C > 0$ depend only  on $\alpha$, $a$, $b$, and $T$. Taking $\varepsilon > 0$ sufficiently small and combining with \eqref{01.03.3}, we obtain
\begin{equation}\label{01.03.5}
\begin{split}
&s \iint_{Q_k^1} \Theta x^\alpha (\partial_x v_k)^2 \, \mathrm{d}x \, \mathrm{d}t + s \iint_{Q_k^2 \cup Q_k^3} \Theta (\partial_x v_k)^2 \, \mathrm{d}x \, \mathrm{d}t \\
&+ s^3 \iint_{Q_k^1} \Theta^3 v_k^2 x^{2 - \alpha} \, \mathrm{d}x \, \mathrm{d}t + s^3 \iint_{Q_k^2 \cup Q_k^3} \Theta^3 v_k^2 \, \mathrm{d}x \, \mathrm{d}t \\
&+ s^{-1} \iint_{Q_k} \Theta^{-1} (|\partial_t v_k|^2 + |\partial_x (w \partial_x v_k)|^2) \, \mathrm{d}x \, \mathrm{d}t \\
&\leq C \|e^{s \xi} g_k\|_{L^2(Q_k)}^2 + C s^3 \iint_{\omega \times (0, T)} \Theta^3 v_k^2 \, \mathrm{d}x \, \mathrm{d}t.
\end{split}
\end{equation}

 Finally, let us consider $\vp_k = e^{-s\xi}v_k$ in equation \eqref{01.03.4}. By applying the product rule for differentiation, we obtain
\begin{equation*}
\begin{split}
\pt_x \vp_k
&= e^{-s\xi}\pt_x v_k + (-s\pt_x\xi)v_ke^{-s\xi}\\
&=
\begin{cases}
e^{-s\xi}\left[\pt_x v_k - (2 - \al)s\ga\Theta x^{1 - \al} v_k\right], & (x,t) \in Q_k^1, \\
e^{-s\xi}\left[\pt_x v_k + s\ga\Theta x^{-\al - 1}(\al + x - \al x)v_k\right], & (x,t) \in Q_k^3.
\end{cases}
\end{split}
\end{equation*}
Consequently, we have the following inequality:
\begin{equation*}\label{01.03.6}
\begin{split}
&s\iint_{Q_k^1}\Theta x^\al (\pt_x\vp_k)^2e^{2s\xi}\df x\df t + s\iint_{Q_k^2 \cup Q_k^3}\Theta (\pt_x\vp_k)^2e^{2s\xi}\df x\df t\\
&+ s^3\iint_{Q_k^1} \Theta^3 \vp_k^2x^{2 - \al}e^{2s\xi}\df x\df t + s^3\iint_{Q_k^2 \cup Q_k^3}\Theta^3 \vp_k^2e^{2s\xi}\df x\df t\\
&\leq C\|e^{s\xi}g_k\|_{L^2(Q_k)}^2 + Cs^3 \iint_{\om \times (0,T)}\Theta^3\vp_k^2e^{2s\xi}\df x\df t,
\end{split}
\end{equation*}
where the positive constant $C$ depends only  on $a$, $b$, $\al$, and $T$.
Based on the above derivation, we now present Theorem \ref{01.03.T1}.
\begin{theorem}\label{01.03.T1}
Let $T > 0$ and $k \in \N$. Then, there exists a positive constant $C$, which depends only on $\al$, $a$, $b$, and $T$, such that for every solution $\vp_k$ of \eqref{01.02.2} and for all $s \geq 1$, the following inequality holds:
\begin{equation}\label{01.03.6}
\begin{split}
&s\iint_{Q_k^1}\Theta x^\al (\pt_x\vp_k)^2e^{2s\xi}\df x\df t + s\iint_{Q_k^2 \cup Q_k^3}\Theta (\pt_x\vp_k)^2e^{2s\xi}\df x\df t\\
&+ s^3\iint_{Q_k^1} \Theta^3 \vp_k^2x^{2 - \al}e^{2s\xi}\df x\df t + s^3\iint_{Q_k^2 \cup Q_k^3}\Theta^3 \vp_k^2e^{2s\xi}\df x\df t\\
&\leq C\|e^{s\xi}g_k\|_{L^2(Q_k)}^2 + Cs^3\iint_{\om \times (0,T)}\Theta^3\vp_k^2e^{2s\xi}\df x\df t.
\end{split}
\end{equation}
\end{theorem}

 \subsection{Approximation}
Based on Theorem \ref{01.01.T2}, we derive the following Theorem \ref{01.03.T2}.
\begin{theorem}\label{01.03.T2}
Let $\omega \subset \Omega$ with $\overline{\omega} \subset \Omega$, and $\varphi_T \in L^2(\Omega)$. Then, there exists a constant $C$, which depends only  on $\alpha$, $a$, $b$, and $T$, such that for any solution $\varphi$ of \eqref{01.02.1} {{and for all $s\geq 1$}}, the following inequality holds:
\begin{equation}\label{01.03.7}
\begin{split}
&s\iint_{Q}\Theta x^\alpha (\partial_x\varphi)^2e^{2s\xi}\,dx\,dt + s^3\iint_{Q} \Theta^3 \varphi^2x^{2-\alpha}e^{2s\xi}\,dx\,dt\\
&\leq C\|e^{s\xi}g\|_{L^2(Q)}^2 + Cs^3\iint_{\omega \times (0,T)}\Theta^3\varphi^2e^{2s\xi}\,dx\,dt.
\end{split}
\end{equation}
\end{theorem}
\begin{proof}
From Theorem \ref{01.01.T2}, for $\varphi_T \in H^1(\Omega)$, we have the weak convergence:
\begin{equation*}
\varphi_k \rightharpoonup \varphi \text{ weakly in } H_0^{1,2}(Q; w) \hbox{ and } \varphi_k \rightharpoonup \varphi \text{ weakly in } L^2(Q; w),
\end{equation*}
and the strong convergence:
\begin{equation*}
\varphi_k \rightarrow \varphi \text{ strongly in } L^2(\omega; w).
\end{equation*}
Given that $g_k \rightarrow g$ in $L^2(Q)$, we obtain \eqref{01.03.7} by taking the limit of \eqref{01.03.6}.
Next, let $\varphi_T^n \in H_0^1(\Omega)$ for $n \in \mathbb{N}$ such that $\varphi_T^n \rightarrow \varphi_T$ in $L^2(\Omega)$. Denote $\varphi_n$ ($n \in \mathbb{N}$) and $\varphi$  as the solutions of the equation \eqref{01.02.1} corresponding to the initial data $\varphi_T^n$ ($n \in \mathbb{N}$) and $\varphi_T$, respectively. Then, $\psi_n = \varphi_n - \varphi$ is the solution of the following equation:
\begin{equation*}
\begin{cases}
\partial_t \psi_n + \partial_x(w\partial_x\psi_n) = 0 &\text{in } Q, \\
\psi_n = 0 &\text{on } \partial Q, \\
\psi_n(T) = \varphi_T^n - \varphi_T &\text{in } L^2(\Omega),
\end{cases}
\end{equation*}
For every $\tau \in (0,T)$, we have
\begin{equation*}
\begin{split}
\frac{1}{2}\int_\Omega \psi_n(x,\tau)^2\,dx + \iint_Q w(\partial_x\psi_n)^2\,dx\,dt = \frac{1}{2}\int_\Omega \psi_n(x,T)^2\,dx \leq \frac{1}{2}\|\varphi_T^n - \varphi_T\|_{L^2(\Omega)}^2.
\end{split}
\end{equation*}
This implies that
\begin{equation*}
\varphi_n \rightarrow \varphi \text{ strongly in } H_0^{1,2}(Q; w),
\end{equation*}
by the compact embedding $H_0^1(\Omega; w) \hookrightarrow L^2(\Omega; w)$ (see \cite{Alabau}). Therefore, for the initial data $\varphi_T \in L^2(\Omega)$, we choose $\varphi_T^n \rightarrow \varphi_T$ in $L^2(\Omega)$ with $\varphi_T^n \in H_0^1(\Omega)$ for $n \in \mathbb{N}$, and by taking the limit, we obtain the desired result.
\end{proof}

{{Now, from Hardy's inequality (given as (2.1) in [1], p. 165, with $a(x)=x^\alpha$), namely,
\begin{equation*}
	\int_0^1 x^{\alpha-2}\varphi^2 \, dx \leq C \int_0^1 x^\alpha (\partial_x \varphi)^2 \, dx,
\end{equation*}
where $C > 0$ depends on $\alpha$ only, we derive:
\begin{equation*}
	\begin{split}
		s \iint_{Q} \Theta \varphi^2 e^{2s\xi} \, dx \, dt
		&\leq s \iint_Q \Theta x^{\alpha-2} \varphi^2 e^{2s\xi} \, dx \, dt \leq s \iint_Q \Theta x^\alpha \left[\partial_x(\varphi e^{s\xi})\right]^2 \, dx \, dt \\
		&= s \iint_Q \Theta x^\alpha \left[e^{s\xi} \partial_x \varphi + s e^{s\xi} \varphi \partial_x \xi\right]^2 \, dx \, dt \\
		&\leq 2s \iint_Q \Theta x^\alpha (\partial_x \varphi)^2 e^{2s\xi} \, dx \, dt + Cs^3 \gamma_0^2 \iint_Q \Theta^3 x^{2-\alpha} \varphi^2 e^{2s\xi} \, dx \, dt \\
		&\leq Cs^3 \gamma_0^2 \iint_{\omega \times (0,T)} \Theta^3 \varphi^2 e^{2s\xi} \, dx \, dt,
	\end{split}
\end{equation*}
where the third inequality uses $x^\alpha (\partial_x \xi)^2 \leq C \gamma_0^2 \Theta^2 x^{2-\alpha}$ for some constant $C > 0$ depending only on $\alpha$, and the last inequality employs (4.18) in Theorem 4.2 with $g=0$. By choosing $s_0 \geq 1$, we obtain:
\begin{equation*}
	\iint_{Q} \Theta \varphi^2 e^{2s_0\xi} \, dx \, dt \leq C \iint_Q \Theta^3 \varphi^2 e^{2s_0\xi} \, dx \, dt.
\end{equation*}
Here, the constant $C > 0$ depends solely on $\alpha$, $\gamma_0$, and $T$. Utilizing the inequalities:
\begin{equation*}
	e^{2s_0\xi} \Theta \geq e^{-C(1+\frac{1}{T})} \frac{1}{T^8} \quad \text{in} \quad (0,1) \times \left(\frac{T}{4}, \frac{3T}{4}\right),
\end{equation*}
and
\begin{equation*}
	e^{2s_0\xi} \Theta^3 \leq e^{-C(1+\frac{1}{T})} \frac{1}{T^8} \quad \text{in} \quad (0,1) \times (0,T),
\end{equation*}
we deduce:
\begin{equation*}
	\iint_{(0,1) \times (\frac{T}{4}, \frac{3T}{4})} |\varphi|^2 \, dx \, dt \leq C \iint_{\omega \times (0,T)} |\varphi|^2 \, dx \, dt,
\end{equation*}
where the constant $C > 0$ depends solely on $\alpha$, $a$, $b$, $\gamma_0$, and $T$. Finally, since:
\begin{equation*}
	\|\varphi(0)\|_{L^2(0,1)}^2 \leq \frac{2}{T} \int_{\frac{T}{4}}^{\frac{3T}{4}} \|\varphi(t)\|_{L^2(0,1)}^2 \, dt,
\end{equation*}
it follows that:
\begin{equation}\label{12.21.1}
	\|\varphi(0)\|_{L^2(0,1)}^2 \leq C \iint_{\omega \times (0,T)} |\varphi|^2 \, dx \, dt,
\end{equation}
which represents the observability inequality, where the constant $C > 0$ depends  on $\alpha$, $a$, $b$, $\gamma_0$, and $T$ only.}}  As established in \cite[Proposition 4.1, Section 4]{Alabau}, the following result holds firmly.
\begin{corollary}
The system \eqref{01.01.E2} is exactly controllable under the control $\chi_\omega f$.
\end{corollary}

\begin{proof}
From Theorem \ref{01.03.T2}, we can deduce   \cite[Lemma 4.2, Section 4, p. 187]{Alabau} in turn, implies   \cite[Proposition 4.1, Section 4]{Alabau}, which is the observability inequality for \eqref{01.01.E2}, {{or \eqref{12.21.1} directly}}. By the Hilbert uniqueness method, we conclude that the system \eqref{01.01.E2} is exactly controllable.
\end{proof}

\begin{remark}\label{Re.4.1} 
	{ It is important to highlight that the control domain  $\omega$ 
 serves as an interior subdomain within  $\Omega$. As a result, the control is applied exclusively to the subdomain
$\omega$, rather than the entire domain $\Omega$. Therefore, it makes no significant difference whether the controls
$f \in L^2(\Omega, w^{-1})$ or $f\in L^2(\Omega)$  as presented in \eqref{01.01.E2}. Indeed, due to Hardy's inequality, the case where
$f \in L^2(\Omega)$  can also be considered.
}
\end{remark}

\section{Concluding Remarks}\label{Se5}

 In this paper, we apply for the first time  the shape design method to approximate solutions for both degenerate elliptic and parabolic equations. As a practical demonstration, we employ this method to approximate the Carleman estimate for a one-dimensional  degenerate parabolic equation. Through this Carleman estimate, we derive the observability of the degenerate parabolic equation in a novel manner. This novel  methodology paves the way for future research aimed at addressing the Carleman estimate or observability of higher-dimensional degenerate parabolic equations.

\end{document}